\documentclass[11pt]{amsart}

\usepackage{amsmath}
\usepackage{amssymb}
\usepackage{amscd}
\usepackage{amsthm}
\usepackage{indentfirst}

\newcommand{\Z}{\ensuremath{\mathbb{Z}}}
\newcommand{\Q}{\ensuremath{\mathbb{Q}}}

\newcommand{\F}{\ensuremath{\mathbb{F}}}

\usepackage[margin=2.6cm]{geometry}

\theoremstyle{plain}
\newtheorem{theorem}{Theorem}[section]
\newtheorem{lemma}[theorem]{Lemma}

\newtheorem{proposition}[theorem]{Proposition}

\theoremstyle{definition}
\newtheorem{remark}[theorem]{Remark}

\DeclareMathOperator{\Spec}{Spec}
\DeclareMathOperator{\Pic}{Pic}
\DeclareMathOperator{\divisor}{div}

\DeclareMathOperator{\Div}{Div}
\DeclareMathOperator{\Gal}{Gal}
\DeclareMathOperator{\Res}{Res}
\DeclareMathOperator{\Jac}{Jac}
\DeclareMathOperator{\Hom}{Hom}
\DeclareMathOperator{\End}{End}

\DeclareMathOperator{\rk}{rk}
\DeclareMathOperator{\car}{char}

\DeclareMathOperator{\GL}{GL}
\DeclareMathOperator{\SL}{SL}
\DeclareMathOperator{\SP}{Sp} 
\DeclareMathOperator{\GSP}{GSp} 
\DeclareMathOperator{\NS}{NS} 

\newcommand{\A}{\mathcal{A}}
\newcommand{\B}{B}
\newcommand{\Bbar}{\overline{B}}

\newcommand{\Gm}{\mathbf{G}_{{\rm m}}}
\newcommand{\X}{\mathcal{X}}
\newcommand{\J}{\mathcal{J}}

\newcommand{\kbar}{\bar{k}}
\newcommand{\Kbar}{\bar{K}}
\newcommand{\Kell}{K_\ell}

\begin{document}

\title{Arithmetic rank bounds for abelian varieties over function fields}

\author{F\'elix Baril Boudreau}
\address{Universit\'e du Luxembourg, D\'epartement de Math\'ematiques, Maison du nombre, 6, avenue de la Fonte, Esch-sur-Alzette, 4364, Luxembourg}
\email{felix.barilboudreau@uni.lu}

\author{Jean Gillibert}
\address{Universit\'e de Toulouse, Institut de Math{\'e}matiques de Toulouse, CNRS UMR 5219, 118 route de Narbonne, 31062 Toulouse Cedex 9, France}
\email{jean.gillibert@math.univ-toulouse.fr}

\author{Aaron Levin}
\address{Department of Mathematics, Michigan State University, 619 Red Cedar Road, East Lansing, MI 48824}
\email{adlevin@math.msu.edu}



\date{September 2025}

\begin{abstract}
It follows from the Grothendieck-Ogg-Shafarevich formula that the rank of an abelian variety (with trivial trace) defined over the function field of a curve is bounded by a quantity which depends on the genus of the base curve and on bad reduction data. Using a function field version of classical $\ell$-descent techniques, we derive an arithmetic refinement of this bound, extending previous work of the second and third authors from elliptic curves to abelian varieties, and improving on their result in the case of elliptic curves. When the abelian variety is the Jacobian of a hyperelliptic curve, we produce a more explicit $2$-descent map.
Then we apply this machinery to studying points on the Jacobians of certain genus $2$ curves over $k(t)$, where $k$ is some perfect base field of characteristic not $2$.
\end{abstract}

\maketitle


\section{Introduction}

It is known since the works of Ogg \cite{ogg62} and Shafarevich \cite{shaf61} (under tameness assumptions), followed by Grothendieck \cite{raynaud95}, that the rank of a given abelian variety over the function field of a curve is bounded by a quantity which depends on the genus of the base curve and on reduction data. More precisely, let $k$ be a perfect field, let $\B$ be a smooth projective geometrically integral $k$-curve of genus $g_\B$, and let $A$ be an abelian variety of dimension $d_A$ over $k(\B)$ whose $k(\B)/k$-trace vanishes
(the trace is, roughly speaking, the largest abelian subvariety of $A$ which is defined over $k$, see \cite[Section~7]{conrad06} for a modern exposition). Then, as a consequence of the Grothendieck-Ogg-Shafarevich formula applied to the $\ell$-torsion subgroup of the N{\'e}ron model of $A$ over the base curve $\B$, we have the inequality
\begin{equation}
	\label{eq:geometricbound}
	\rk_\Z A(k(\B)) \leq 2d_A \cdot (2g_\B-2) + \deg(\mathfrak{f}_A),
\end{equation}
where $\mathfrak{f}_A$ is the conductor of $A$, which is a divisor on $\B$ (see \S{}\ref{subsec:comparison} for details).

Note that the right-hand side of \eqref{eq:geometricbound} does not change if $k$ is replaced by its algebraic closure $\kbar$.
Following the terminology introduced by Ulmer \cite{ulmer04}, we call \eqref{eq:geometricbound} the \emph{geometric rank bound}, and any rank bound that is sensitive to the base field is said to be an \emph{arithmetic rank bound}.

When $A$ is an elliptic curve and $k$ is a finite field, an arithmetic rank bound depending on the cardinality of $k$ was given by Brumer \cite[Proposition~6.9]{Brumer92}; it improves on the geometric one provided $|k|< \deg(\mathfrak{f}_A)^2$. In \cite{gl22} we obtained, for elliptic curves satisfying a mild assumption, an arithmetic refinement of the geometric bound, giving an answer to a question raised by Ulmer \cite[Section 9]{ulmer04}.
The aim of this paper is to prove the existence of similar arithmetic bounds for higher dimensional abelian varieties, and to explore various consequences of these results.

Our main result is the following.

\begin{theorem}
\label{thm:main}
Let $k$ be a perfect field and let $\B$ be a smooth projective geometrically integral $k$-curve.
Let $A$ be an abelian variety over $k(\B)$ whose $k(\B)/k$-trace vanishes, and let $\ell$ be a prime number invertible in $k$.

(i) We have
\begin{equation}
\label{eq:arithmeticbound}
\rk_\Z A(k(\B)) \leq \dim_{\F_\ell} H^1(\B,\A[\ell]) + \dim_{\F_\ell} H^0(\B,\Phi/\ell\Phi),
\end{equation}
where $\A[\ell]$ denotes the $\ell$-torsion subgroup of the N\'eron model $\A\to \B$ of $A$, $\Phi$ denotes its group of components $\A/\A^0$ and $\Phi/\ell\Phi$ denotes the cokernel of multiplication by $\ell$ on $\Phi$. The cohomology groups are computed in the {\'e}tale topology.

(ii) If in addition both $A$ and its dual abelian variety $A^t$ have no nonzero rational $\ell$-torsion points over $\kbar(\B)$, then this bound is Galois equivariant in the sense that the cohomology groups on the right-hand side are the $\Gal(\kbar/k)$-invariants of their counterparts computed over $\kbar$, and it constitutes a refinement of \eqref{eq:geometricbound} in the sense that, when $k=\kbar$, it coincides with the geometric rank bound $2d_A \cdot (2g_\B-2) + \deg(\mathfrak{f}_A)$.
\end{theorem}

To summarize, \eqref{eq:arithmeticbound} is an arithmetic rank bound which,
under the assumptions of Theorem~\ref{thm:main}~(ii), refines the geometric one \eqref{eq:geometricbound}. On the other hand, when $A$ has nonzero rational $\ell$-torsion points then the bound \eqref{eq:arithmeticbound} may be infinite, see Remark~\ref{remark:H0and1}.

We note that our arithmetic rank bound depends on $\ell$, while the geometric rank bound does not. It would therefore be more precise to say that \eqref{eq:arithmeticbound} is a \emph{family} of arithmetic rank bounds indexed by the primes $\ell$ satisfying the assumptions of Theorem~\ref{thm:main}~(ii).

It is an immediate consequence of the Lang-Néron theorem that $A$ and $A^t$ have finitely many rational torsion points over $\kbar(\B)$. Therefore, the assumptions of Theorem~\ref{thm:main}~(ii) are satisfied for all but finitely many primes $\ell$.

We note that, already in the case of elliptic curves, Theorem~\ref{thm:main} constitutes an improvement to the main result of \cite{gl22}:
\begin{itemize}
	\item the assumptions on $A[\ell]$ are less restrictive: in \emph{loc. cit.} we assume that the absolute Galois group of $\kbar(\B)$ acts transitively on $A[\ell]\setminus\{0\}$;
	\item the bound \eqref{eq:arithmeticbound} is a refinement of the geometric bound for any value of $\ell$ satisfying the assumptions, whereas in \cite{gl22} this is proved only for $\ell=2$, and for $\ell = 3$ when $k$ does not contain third roots of unity.
\end{itemize}

Since the geometric rank bound \eqref{eq:geometricbound} is given by an explicit formula, one may ask how explicit the arithmetic rank bound \eqref{eq:arithmeticbound} can be made, say, when $A$ is the Jacobian of some curve $C$. Given an equation of $C$, one first computes a regular model of $C$ over $\B$ by performing blow-ups. Then, from the data of the intersection graph of the components of the special fiber of this regular model at a bad place $v$, one computes the local groups of components $\Phi_v$ (see \cite[Section 9.6]{NeronModels}). With this procedure, $H^0(\B,\Phi/\ell\Phi)$ can be made explicit. Things are less clear for the group $H^1(\B,\A[\ell])$. In Theorem~\ref{thm:big_monodromy}, we prove, under ``big monodromy'' assumptions, the existence of an injective map
$$
H^1(\B,\A[\ell]) \hookrightarrow \ker(N_{\X/\X^+}:\Pic(\X)[\ell] \to \Pic(\X^+)[\ell]),
$$
where $\X\to \B$ is the (geometrically  irreducible) cover of smooth projective $k$-curves with generic fiber $A[\ell]\setminus\{0\}$, and $\X\to \X^+$ is the quotient of $\X$ by the involution $P\mapsto -P$. In the case of elliptic curves, we stated in \cite[Theorem~1.1]{gl22} the resulting upper bound on the dimension of $H^1(\B,\A[\ell])$. The proof we give here is along the same lines as in the elliptic case. When computed over the algebraic closure of $k$, this bound is usually not as good as the geometric bound, since the genus of $\X$ increases with $\ell$, hence the interest of this statement is arithmetic. Nevertheless, in the particular case when $A$ is the Jacobian of a hyperelliptic curve and $\ell=2$, we obtain the following statement.

\begin{theorem}
	\label{thm:jacobian}
	Assume $\car(k)\neq 2$, and
	let $J$ be the Jacobian of a hyperelliptic curve defined over $k(\B)$ by an equation of the form
	$$
	y^2=f(x),
	$$
	where $f\in k(\B)[x]$ is irreducible over $\overline{k}(\B)$, of odd degree. Let $\X$ be the smooth projective $k$-curve defined by the equation $f(x)=0$. Then we have
	$$
	\rk_\Z J(k(\B)) \leq \dim_{\F_2} \Pic(\X)[2] - \dim_{\F_2} \Pic(\B)[2] + \dim_{\F_2} H^0(\B,\Phi/2\Phi),
	$$
	and this bound agrees with the geometric rank bound \eqref{eq:geometricbound} when $k$ is algebraically closed.
\end{theorem}

In the setting of Theorem~\ref{theoremd5} below, the above bound is sharp, as one can see from the proof of Theorem~\ref{theoremd5}~(i).

The statement of Theorem~\ref{thm:jacobian} can be made more explicit. Firstly, we introduce some notation. If $Y$ is a curve and $\Sigma$ is a reduced divisor on $Y$, the group of divisors with rational coefficients above $\Sigma$ is defined to be
$$
\Div(Y\setminus \Sigma) \oplus \bigoplus_{P\in\Sigma} \Q\cdot P.
$$
We denote by $\Pic(Y,\Q.\Sigma)$ the quotient of this group by the usual group of principal divisors on $Y$; we refer the reader to \cite[\S{}2]{GHL23} for basic properties of this construction.

Let $C$ be a hyperelliptic curve over $k(\B)$ defined by an equation of the form $y^2=f(x)$, with the same assumptions as in Theorem~\ref{thm:jacobian}. Since $f$ has odd degree, $C$ has a unique point at infinity. A semi-reduced divisor (of degree zero) on $C$ is by definition of the form $\sum n_iP_i -(\sum n_i)\infty$ where the $P_i$ are points on the affine model, whose $x$-coordinates are all distinct, and the $n_i$ are positive integers; one also requires ramification points of $x:C\to \mathbb{P}^1$ to appear with multiplicity at most one. The Mumford representation of a semi-reduced divisor $D$ of degree zero on $C$ is a pair $(a_D,b_D)$, where $a_D\in k(B)[x]$ denotes the monic polynomial whose zeroes are the $x$-coordinates of affine points of the divisor $D$, taking into account multiplicities, and $b_D\in k(B)[x]$ is the polynomial whose value at the $x$-coordinate of a point of $D$ is the corresponding $y$-coordinate, and such that $\deg b_D < \deg a_D$ (one can construct it by Lagrange interpolation). We refer to the original work of Mumford \cite[Chapter~IIIa]{tata2} for further details.

\begin{proposition}
	\label{explicitdescent}
	Under the assumptions of Theorem~\ref{thm:jacobian}, we have an injective morphism
	\begin{align*}
		\phi_2:J(k(\B))/2J(k(\B)) &\hookrightarrow \Pic(\X,\Q.\Sigma_2^\X)[2] \\
		[D] &\mapsto \frac{1}{2}\divisor(a_D(x)).
	\end{align*}
	In this formula, $D$ is a semi-reduced divisor, $a_D$ is the $x$-coordinate polynomial in the Mumford representation of $D$, and $\Sigma_2^\X$ is the inverse image in $\X$ of the set of places $v\in \B$ of bad reduction of $J$ for which $(\Phi_v/2\Phi_v)(k_v)\neq 0$, where $k_v$ denotes the residue field of $v$.
\end{proposition}

When $f$ has degree $3$, the curve $y^2=f(x)$ is elliptic, and the properties of the map $\phi_2$ above were proved in \cite[\S{}4]{GHL23}. The general case can be proved along the same lines, thanks to the work of Schaefer \cite[Theorem~1.2]{Schaefer1995}, who gives an explicit description of the 2-descent map for Jacobians of hyperelliptic curves. The details are given in \S{}\ref{subsec:hyperelliptic}.

The main interest of introducing the 2-descent map $\phi_2$ is that one can compute its values explicitly.
The second and third authors with Hallouin \cite{GHL23} used this tool in order to study integral points on elliptic curves over $k(t)$. Continuing this work, we now study in detail the $k(t)$-points of the Jacobians of genus $2$ curves $C$ over $k(t)=k(\mathbb{P}^1)$ of the form: 
\begin{align*}
	C:y^2=x^d+g(t),
\end{align*}
where $d=5,6$, $\car(k)\nmid 6d$, and $g(t)$ is squarefree with $\deg g(t)=3$ (the cases $\deg g(t)<3$ could be handled similarly).  In particular, we study certain subsets of $k(t)$-points with small height and satisfying certain integrality properties. The assumptions on $d$ and $g$ imply that the underlying surface is rational, which is a quite advantageous situation.

Suppose first that $d=5$.   Let $J$ be the Jacobian of $C$ over $k(t)$. For a divisor class in $J(k(t))$ represented by a (semi-reduced) divisor $D$, we let $(a_D,b_D)\in k(t)[x]^2$ be the pair of polynomials from the Mumford representation of $D$ as recalled before Proposition~\ref{explicitdescent}, and we let $c_D\in k(t)[y]$ be the polynomial whose roots are the $y$-coordinates of the (affine) points of $D$, counted with multiplicity.

As before, let $\mathcal{X}$ be the smooth projective curve over $k$ defined by $x^5+g(t)=0$, and let $\divisor(x)=P_1+P_2+P_3-3\infty$, where $P_1,P_2,P_3$ correspond to the $3$ roots of $g$ and $\infty$ denotes the unique point on $\mathcal{X}$ at infinity. Applying the Riemann-Hurwitz formula to the projection $\mathcal{X}\to\mathbb{P}^1$, $(t,x)\mapsto t$, and using that $P_1+P_2+P_3\sim 3\infty$, we find that
\begin{align*}
K_{\mathcal{X}}\sim 5(-2\infty)+4(\infty+P_1+P_2+P_3)\sim 6\infty.
\end{align*}
In particular, the genus of $\mathcal{X}$ is $4$. Recall that a theta characteristic on $\mathcal{X}$ is a divisor class $[D]$ such that $2D\sim K_{\mathcal{X}}$, and it is called odd (resp.\ even) if $\dim H^0(\mathcal{X},\mathcal{O}(D))$ is odd (resp.\ even). It is well known \cite{Mumford71} that the number of odd theta characteristics on a curve of genus $g$ over $\kbar$ is given by $2^{g-1}(2^{g}-1)$ (assuming as always that $\car(k)\neq 2$).

\begin{theorem}
	\label{theoremd5}
	Assume that $\car(k)\nmid 30$.
	Let $J$ be the Jacobian of a genus $2$ curve $C$ defined over $k(t)$  by an equation of the form
	\begin{align*}
		C:y^2=x^5+g(t),
	\end{align*}
	where $g(t)$ is a squarefree cubic polynomial over $k$.  Let
	\begin{align*}
		R:=\{[D]\in J(\kbar(t))\mid a_D=x^2+bx+c(t), b\in \kbar, c(t)\in \kbar[t], \deg c(t)=1\}.
	\end{align*}
	\begin{enumerate}
	\item[(i)] The group $J(\kbar(t))$ is torsion-free, the points of $R$ generate $J(\kbar(t))$, and the Mordell-Weil lattice of $J(\kbar(t))$ is the $E_8$ lattice, with the elements of $R$ corresponding to the $240$ root vectors of $E_8$.\label{theoremd5i}
	\item[(ii)] There is a $2$-to-$1$ map (induced by $2$-descent)
	\begin{align*}
		\{[D]\in R\mid a_D\in k(t)[x]\}&\to \{\text{odd $k$-rational theta characteristics on $\mathcal{X}$} \},\\
		[D]&\mapsto \frac{1}{2}\divisor(a_D)+3\infty.
	\end{align*}
	In particular, when $k=\kbar$ there is a $2$-to-$1$ map, $R\to \{\text{odd theta characteristics on $\mathcal{X}$} \}$.
	
	\item[(iii)]
	Let $E$ be the elliptic curve over $k$ given by $y^2=g(t)$.  There is a $10$-to-$1$ map (induced by $5$-isogeny descent)
	\begin{align*}
		\{[D]\in R\mid c_D\in k(t)[y]\}&\to E[5](k)\setminus\{0\},\\
		[D]&\mapsto \frac{1}{5}\divisor(c_D).
	\end{align*}
		In particular, when $k=\kbar$ there is a $10$-to-$1$ map, $R\to  E[5](\kbar)\setminus\{0\}$.

	\end{enumerate}
	From (ii) and (iii) with $k=\kbar$, it follows that the cardinality of $R$ is given by
	$$
		|R|=240 = 2^8-2^4 =2(5^3-5).
	$$

\end{theorem}
We note that using two different descent maps to study the set $R$ allows one to give an ``explanation" of a striking exponential-Diophantine identity.  Indeed, dividing the last identities of the theorem by $2$, we find
\begin{align}
	120=2^7-2^3=5^3-5.\label{unit1}
\end{align}

We note the further identity 
\begin{align}
	240=2^8-2^4=3^5-3\label{unit2},
\end{align}
which can be obtained by arguments analogous to those used in the proof of Theorem~\ref{theoremd5} applied to certain elliptic surfaces (see \cite[Table~1]{GHL23}). Similar considerations on elliptic surfaces in \cite{GHL23} produced the identities:
\begin{align}
	6&=2^3-2=3^2-3,\label{unit3}\\
	24&=2^5-2^3=3^3-3\label{unit4}.
\end{align}
Alternatively, each of \eqref{unit1}-\eqref{unit4} produces two solutions to an equation of the form
\begin{align}
	\label{Pillai}
	2^x-p^y=c
\end{align}
in positive integers $x,y$, where $c$ is a fixed nonzero integer and $p$ a fixed (positive) prime. Such equations were systematically studied by Pillai in the 1930's and 1940's.  Remarkably, it is conjectured that \eqref{unit1}-\eqref{unit4} yield all cases where \eqref{Pillai} has two solutions (see \cite{Bennett01} for a more general conjecture). In the case $p=3$, this was conjectured by Pillai \cite{Pillai} and proven by Stroeker and Tijdeman \cite{ST}. Thus, by using two different descent maps to study certain sets of points on Jacobians over function fields, we find ``explanations" of (conjecturally) every instance of two solutions to Pillai's equation \eqref{Pillai}.

We now turn to the case $d=6$. Since $d$ is even, Theorem~\ref{thm:jacobian} does not apply, and indeed we will see that this case exhibits quite different behavior; in particular, the natural analogue of the descent map $\phi_2$ of Proposition~\ref{explicitdescent} is not even well-defined (see Remark \ref{remdescent}).

\begin{theorem}
	\label{theoremd6}
	Assume that $\car(k)\nmid 6$.
	Let $J$ be the Jacobian of a genus $2$ curve $C$ defined over $k(t)$ by an equation of the form
	\begin{align*}
		C:y^2=x^6+g(t),
	\end{align*}
	where $g(t)$ is a squarefree cubic polynomial over $k$. Then
	\begin{align*}
		J(\kbar(t))\cong \mathbb{Z}^6\oplus \mathbb{Z}/3\mathbb{Z},
	\end{align*}
	and the associated Mordell-Weil lattice is $E_6^*$ (the dual of the $E_6$ lattice). The torsion subgroup is cyclic of order three, generated by $[\infty^--\infty^+]$ where $\infty^+,\infty^-$ are the two points at infinity on $C$ (in some order).  Let
	\begin{align*}
		R:=\{[D]\in J(\kbar(t))\mid a_D=x^2+bx+c(t) \text{ or }a_D=x+c(t), b\in \kbar, c(t)\in \kbar[t], \deg c(t) = 1\}.
	\end{align*}
	
	The points of $R$ generate $J(\kbar(t))$ modulo torsion. We may write $R=R_1\cup R_2$, where $R_1$ is mapped $3$-to-$1$ onto the $54$ elements of minimal norm $4/3$ in $E_6^*$ (the $3$-torsion subgroup acts on $R_1$) and $R_2$  is mapped $1$-to-$1$ onto the $72$ elements of norm $2$ in $E_6^*$. We have
	\begin{align*}
		|R|=234=3\cdot 54+72.
	\end{align*}
\end{theorem}


\section{Proofs}
\label{sec:main}


Throughout this paper, all cohomology groups are computed with respect to the {\'e}tale topology, the Galois cohomology being seen as a special instance.

We consider the following setting:

\begin{enumerate}
	\item $k$ is a perfect field, $\B$ is a smooth projective geometrically integral $k$-curve of genus $g_\B$, and $\ell\neq \car(k)$ is a prime number. We denote by $k(B)$ the function field of $B$. Given a place $v \in \B$, we denote by $k_v$ its residue field. Finally, we let $\overline{\B} := B \times_{\Spec(k)} \Spec(\kbar)$.
	\item $A$ is an abelian variety of dimension $d_A$ over $k(\B)$ whose $k(\B)/k$-trace vanishes.
	\item $\A\to \B$ is the N\'eron model of $A$ over $\B$.
	\item $\Sigma\subset \B$ is the set of places of bad reduction of $\A$.
	\item $\A[\ell]\to \B$ is the group scheme of $\ell$-torsion points of $\A$. The map $\A[\ell]\to \B$ is {\'e}tale, and its restriction to $\B\setminus \Sigma$ is finite of degree $\ell^{2d_A}$.
\end{enumerate}

If $v\in \Sigma$ is a place of bad reduction of $\A$, we denote by $\A_v:=\A\times_\B \Spec(k_v)$ the special fiber of $\A$ at $v$, by $\A_v^0$ the connected component of the identity in $\A_v$, and by $\Phi_v$ the component group of $\A_v$. By definition, we have an exact sequence for the {\'e}tale topology on $k_v$
\begin{equation}
	\label{eq:compgroup}
	\begin{CD}
		0 @>>> \A_v^0 @>>> \A_v @>>> \Phi_v @>>> 0, \\
	\end{CD}
\end{equation}
and $\Phi_v$ is a finite {\'e}tale group scheme over $k_v$.

By globalizing \eqref{eq:compgroup}, we obtain an exact sequence for the {\'e}tale topology on $\B$
$$
\begin{CD}
	0 @>>> \A^0 @>>> \A @>>> \Phi:=\bigoplus_{v\in \Sigma} (i_v)_*\Phi_v @>>> 0, \\
\end{CD}
$$
where $i_v:\Spec(k_v)\hookrightarrow B$ denotes the canonical inclusion.  By construction, $\A^0$ is the open subgroup of $\A$ whose fiber at each $v$ is $\A_v^0$. It is called the \emph{identity component} of $\A$.


\subsection{$\ell$-descent}
\label{subsec:l-descent}


\begin{lemma}\label{lem:l-descent}
	If the $k(B)/k$-trace of $A$ vanishes, then
	$$
	\rk_\Z A(k(\B)) \leq \dim_{\F_\ell} H^1(\B,\A[\ell]) + \dim_{\F_\ell} H^0(\B,\Phi/\ell\Phi).
	$$
\end{lemma}

\begin{proof}
This follows from the same arguments as in \cite[\S{}2.1]{gl22}. For the reader's convenience, we sketch the main steps.

By the universal property of the N{\'e}ron model, $\A(\B)/\ell\A(\B)=A(k(\B))/\ell A(k(\B))$. Since the trace of $A$ vanishes, it follows from the Lang-Néron theorem that $A(k(\B))$ is finitely generated, hence
$$
\rk_\Z A(k(\B)) \leq \dim_{\F_\ell} \A(\B)/\ell\A(\B).
$$
We now give an upper bound on the quantity on the right.

If $\A^{\ell\Phi}$ denotes the inverse image of $\ell\Phi$ by the quotient map $\A\to\Phi$, then we have by construction an exact sequence for the {\'e}tale topology on $\B$:
\begin{equation*}
\begin{CD}
0 @>>> \A^{\ell\Phi} @>>> \A @>>> \Phi/\ell\Phi @>>> 0. \\
\end{CD}
\end{equation*}
Applying global sections to this sequence, we obtain an exact sequence of abelian groups
$$
\begin{CD}
0 @>>> \A^{\ell\Phi}(\B) @>>> \A(\B) @>>> H^0(\B,\Phi/\ell\Phi). \\
\end{CD}
$$

Since $\ell\A(\B)$ is a subgroup of $\A^{\ell\Phi}(\B)$, modding out the first two groups by $\ell\A(\B)$ yields another exact sequence
\begin{equation}
\label{Ep_phi2}
\begin{CD}
0 @>>> \A^{\ell\Phi}(\B)/\ell\A(\B) @>>> \A(\B)/\ell\A(\B) @>>> H^0(\B,\Phi/\ell\Phi). \\
\end{CD}
\end{equation}

Finally, the map $[\ell]:\A^0_v\to \A^0_v$ is surjective for the {\'e}tale topology on $k_v$ (multiplication by $\ell$ on a smooth connected group scheme in characteristic $\neq \ell$ is {\'e}tale surjective). It follows that multiplication by $\ell$ on $\A$ induces an exact sequence for the {\'e}tale topology on $\B$
$$
\begin{CD}
0 @>>> \A[\ell] @>>> \A @>[\ell]>> \A^{\ell\Phi} @>>> 0. \\
\end{CD}
$$
Applying {\'e}tale cohomology to this sequence, we obtain an injective map
$$
\A^{\ell\Phi}(\B)/\ell\A(\B) \hookrightarrow H^1(\B,\A[\ell]).
$$
Combining this with \eqref{Ep_phi2} yields the upper bound on $\dim_{\F_\ell} \A(\B)/\ell\A(\B)$ we were looking for.
\end{proof}


\subsection{Galois equivariance of $H^1$}


In this section we shall prove, under the assumptions of Theorem~\ref{thm:main}~(ii), a Galois-equivariance property for $H^1(\B,\A[\ell])$.

\begin{lemma}
	\label{lem:H1GalInv}
	Assume that $H^0(\overline{\B},\A[\ell])=0=H^2(\overline{\B},\A[\ell])$. Then
	\begin{equation}
		\label{eq:H1GalInv}
		H^1(\B,\A[\ell]) = H^1(\overline{\B},\A[\ell])^{\Gal(\kbar/k)}
	\end{equation}
	and this group is a finite-dimensional $\F_\ell$-vector space.
\end{lemma}

\begin{proof}
	Consider the Leray spectral sequence
	$$
	H^r(k,H^s(\overline{\B},\A[\ell])) \Rightarrow H^{r+s}(\B,\A[\ell]).
	$$
	
	The curve $\B$ being projective, and the sheaf $\A[\ell]$ being constructible, the groups $H^s(\overline{\B},\A[\ell])$ are finite-dimensional  $\F_\ell$-vector spaces, and are zero for $s>2$. Under the assumption that $H^0(\overline{\B},\A[\ell])=0=H^2(\overline{\B},\A[\ell])$, the above spectral sequence degenerates. This implies \eqref{eq:H1GalInv}, from which the last part of the statement follows.
\end{proof}

\begin{remark}
	Under the assumption of Lemma~\ref{lem:H1GalInv}, it also follows from the above Leray spectral sequence that we have
	$$
	H^2(\B,\A[\ell]) = H^1(k,H^1(\overline{\B},\A[\ell])).
	$$
\end{remark}

\begin{lemma}
	\label{lem:CohCond}
	The following conditions are equivalent:
	\begin{enumerate}
		\item[(i)] $H^0(\overline{\B},\A[\ell])=0=H^2(\overline{\B},\A[\ell])$;
		\item[(ii)] $H^0(\overline{\B},\A[\ell])=0=H^0(\overline{\B},\A^t[\ell])$;
		\item[(iii)] $A[\ell]^G =0= A[\ell]_G$, where $G=\Gal(\kbar(B)(A[\ell])/\kbar(\B))$ is the Galois group of the field of definition of $\ell$-torsion points of $A$.
	\end{enumerate}
\end{lemma}

\begin{proof}
	The group schemes $\A[\ell]$ and $\A^t[\ell]$ are N\'eron models of their generic fibers, and are Cartier dual finite {\'e}tale group schemes over the good reduction locus of $\A\to \overline{\B}$. Therefore we have \cite[p.~181, Proposition~2.2(b)]{milne80} a perfect duality of finite-dimensional $\F_\ell$-vector spaces
	$$
	H^0(\overline{\B},\A^t[\ell]) \times H^2(\overline{\B},\A[\ell]) \rightarrow \Z/\ell\Z.
	$$
	This yields a canonical isomorphism
	\begin{equation}
		\label{eq:PoincareDuality}
		H^2(\overline{\B},\A[\ell]) \simeq H^0(\overline{\B},\A^t[\ell])^\vee
	\end{equation}
	where $M^\vee$ denotes the $\F_\ell$-dual of a finite dimensional $\F_\ell$-vector space $M$. Therefore, (i) is equivalent to (ii).
	
	Now, let $G=\Gal(\kbar(B)(A[\ell])/\kbar(\B))$. By the N\'eron mapping property we have that
	$$
	H^0(\overline{\B},\A[\ell]) = A[\ell]^G
	$$
	and similarly for $A^t[\ell]$.
	Since $\ell$ is invertible in $k$, the choice of an $\ell$-th root of unity induces an isomorphism $\mu_\ell\simeq \Z/\ell\Z$ over $\kbar$. Combining this with Cartier duality yields a perfect pairing
	$$
	A[\ell]\times A^t[\ell] \to \Z/\ell\Z
	$$
	which is $G$-equivariant. In particular, we have an isomorphism
	$$
	A^t[\ell]^G\simeq \Hom(A[\ell],\Z/\ell\Z)^G.
	$$
	It follows from Lemma~\ref{lem:inv-and-coinv} that
	$$
	A^t[\ell]^G\simeq (A[\ell]_G)^\vee.
	$$
	Combining this with \eqref{eq:PoincareDuality}, we obtain an isomorphism
	$$
	H^2(\overline{\B},\A[\ell]) \simeq A[\ell]_G,
	$$
	which is canonical (up to the choice of a $\ell$-th root of unity). Thus, (i) and (iii) are equivalent.
\end{proof}

\begin{lemma}
	\label{lem:inv-and-coinv}
	Let $M$ be a finite dimensional $\F_\ell$-vector space $M$ endowed with an action of some group $G$. Then $(M^\vee)^G\simeq (M_G)^\vee$.
\end{lemma}

\begin{proof}
	The functor $M\mapsto M^\vee$ induces a self-duality of the category of finite dimensional $\F_\ell$-vector spaces. The same holds in the category of $\F_\ell[G]$-modules. Since $V^G\to V$ and $V\to V_G$ are solutions of dual universal problems, they are therefore exchanged by the functor $M\mapsto M^\vee$.
\end{proof}

\begin{remark}
	As pointed out in the introduction, under the assumption that the trace of $A$ vanishes, it follows from the Lang-Néron theorem that condition (ii) of Lemma~\ref{lem:CohCond} holds for all but finitely many $\ell$.
\end{remark}


\subsection{Proof of Theorem~\ref{thm:main}}
\label{subsec:comparison}


Since the $\kbar(\B)/\kbar$-trace of $A$ vanishes, the Grothendieck-Ogg-Shafarevich formula (see \cite[Th\'eor\`eme 3~(ii)]{raynaud95}) reads
$$
\rk_\Z A(\kbar(\B)) + r_0 = -\chi(\overline{\B},\A^0[\ell]),
$$
where $\chi(\overline{\B},\A^0[\ell])$ is the {\'e}tale Euler-Poincaré characteristic of the constructible sheaf $\A^0[\ell]$, and $r_0$ is the corank of the $\ell$-primary part of $H^1(\overline{\B},\A)$. These numbers do not depend on $\ell$. The integer $r_0$ being non-negative, one deduces that
\begin{equation}
	\label{eq:GOSbound}
	\rk_\Z A(\kbar(\B)) \leq -\chi(\overline{\B},\A^0[\ell]),
\end{equation}
which is the geometric bound \eqref{eq:geometricbound}. We shall now prove that, under the assumptions of Theorem~\ref{thm:main}~(ii), this bound is equal to the arithmetic one when $k$ is algebraically closed.

\begin{lemma}
	\label{lem:boundcomparison}
	Assume $H^0(\overline{\B},\A[\ell])=0=H^2(\overline{\B},\A[\ell])$. Then we have
	\begin{equation}
		\label{eq:boundcomparison}
		-\chi(\overline{\B},\A^0[\ell]) = \dim_{\F_\ell} H^1(\overline{\B},\A[\ell]) + \dim_{\F_\ell} H^0(\overline{\B},\Phi/\ell\Phi).
	\end{equation}
\end{lemma}

\begin{proof}
	The $\Phi_v$ are finite constant group schemes over $\kbar$. Since we have, for any finite abelian group $M$, a (non-canonical) isomorphism $M/\ell M \simeq M[\ell]$, one deduces that
	$$
	\dim_{\F_\ell} H^0(\overline{\B},\Phi/\ell\Phi) = \dim_{\F_\ell} H^0(\overline{\B},\Phi[\ell]).
	$$
	Therefore, \eqref{eq:boundcomparison} is equivalent to
	$$
	-\chi(\overline{\B},\A^0[\ell]) = \dim_{\F_\ell} H^1(\overline{\B},\A[\ell]) + \dim_{\F_\ell} H^0(\overline{\B},\Phi[\ell]).
	$$
	
	On the other hand, the exact sequence of {\'e}tale sheaves
$$
\begin{CD}
0 @>>> \A^0[\ell] @>>> \A[\ell] @>>> \Phi[\ell] @>>> 0 \\
\end{CD}
$$
	yields
	\begin{equation}
		\label{eq:EulerRelation}
		\chi(\overline{\B},\A^0[\ell]) - \chi(\overline{\B},\A[\ell]) + \chi(\overline{\B},\Phi[\ell]) = 0.
	\end{equation}
	Since $\Phi[\ell]$ is a skyscraper sheaf, we have
	$$
	\chi(\overline{\B},\Phi[\ell]) = \dim_{\F_\ell} H^0(\overline{\B},\Phi[\ell]).
	$$
	Since $\overline{\B}$ is a smooth curve over an algebraically closed field and $\A[\ell]$ is an $\ell$-torsion constructible sheaf over $\overline{\B}$, it is well known that $H^i(\overline{\B},\A[\ell]) = 0$ for $i > 2$ (see \cite[p.221, VI.1.1]{milne80}). Combining this with the assumption $H^0(\overline{\B},\A[\ell])=0=H^2(\overline{\B},\A[\ell])$, it follows that
	$$
	- \chi(\overline{\B},\A[\ell]) = \dim_{\F_\ell} H^1(\overline{\B},\A[\ell]),
	$$
	which concludes the proof.
\end{proof}

\begin{proof}[Proof of Theorem~\ref{thm:main}]
	The bound \eqref{eq:arithmeticbound} holds according to Lemma~\ref{lem:l-descent}. Note that $H^0(\B,\Phi/\ell\Phi)$ is finite and is equal to the $\Gal(\kbar/k)$-invariants of its counterpart computed over $\kbar$.
	
	Assume now that $A$ and its dual abelian variety $A^t$ have no rational $\ell$-torsion points over $\kbar(\B)$. Then according to Lemma~\ref{lem:CohCond}, we have that $H^0(\overline{\B},\A[\ell])=0=H^2(\overline{\B},\A[\ell])$. It then follows from Lemma~\ref{lem:H1GalInv} that $H^1(\B,\A[\ell])$ is finite and is equal to the $\Gal(\kbar/k)$-invariants of its counterpart computed over $\kbar$. Finally, Lemma~\ref{lem:boundcomparison} proves that the bound \eqref{eq:arithmeticbound} is a refinement of the geometric bound, which concludes the proof.
\end{proof}

\begin{remark}
\label{remark:H0and1}
 \begin{enumerate}
 \item Assuming that $A$ and $A^t$ have trivial $\kbar(B)$-rational $\ell$-torsion subgroups is important in the proof of Theorem~\ref{thm:main}. Indeed, if $k$ is a number field, the group $H^1(\B,\A[\ell])$ may be infinite as observed in \cite[Remark~1.5]{gl22}.
 \item If $k$ is a finite field, then $H^1(\B,\A[\ell])$ is closely related to the $\ell$-Selmer group of $A$ over $k(B)$. See for example \cite{cesnavicius16}.
 \item In the case when $A$ is an elliptic curve, one proves \cite[Lemma 2.1]{gl22} that
	$$
	\dim_{\F_\ell} H^0(\B,\Phi/\ell\Phi) = \dim_{\F_\ell} H^0(\B,\Phi[\ell]).
	$$
	This is due to the particular structure of the $\Phi_v$, whose underlying abelian group is either cyclic, or isomorphic to $(\Z/2\Z)^2$. In either case, there exists an isomorphism $\Phi_v/\ell\Phi_v \simeq \Phi_v[\ell]$ which respects the Galois action.
 \end{enumerate}
\end{remark}


\subsection{Cohomology of $\A[\ell]$ and Picard groups of curves}


Let $K$ be a field, and let $(A,\lambda)$ be an abelian variety $A$ of dimension $d_A$ defined over $K$, together with a polarization $\lambda: A \to A^t$ defined over $K$.

For each prime number $\ell \neq \car(K)$ which does not divide the degree of $\lambda$, the composition of the pairing induced by Cartier duality between $A[\ell]$ and $A^t[\ell]$ and the polarization $\lambda$ yields a nondegenerate, alternating and $\Gal(\Kbar/K)$-equivariant pairing
$$
e_\ell^\lambda : A[\ell] \times A[\ell] \to \mu_\ell.
$$
This gives rise to an injective Galois representation
\begin{equation}\label{ell-adic_representation}
\rho_\ell : \Gal(K(A[\ell])/K) \hookrightarrow \GSP_{2d_A}(\F_\ell),
\end{equation}
where $\GSP_{2d_A}(\F_\ell)$ denotes the group of symplectic similitudes. If we let $\SP_{2d_A}(\F_\ell)$ be the corresponding symplectic group, then we have a short exact sequence of groups
$$
\begin{CD}
	1 @>>> \SP_{2d_A}(\F_\ell) @>>> \GSP_{2d_A}(\F_\ell) @>\mathrm{mult.}>> \F_\ell^\times @>>> 1,
\end{CD}
$$
and it is well known that the composition of $\rho_\ell$ with the multiplier map above is the cyclotomic character. In particular, if $K$ contains $\mu_\ell$, then the representation $\rho_\ell$ factors through $\SP_{2d_A}(\F_\ell)$.

We say that $A$ has \emph{big monodromy} if, for $\ell$ large enough, the image of $\rho_\ell$ contains the symplectic group $\SP_{2d_A}(\F_\ell)$. When $K=k(\B)$ is the function field of our curve $\B$, we say that $A$ has \emph{big geometric monodromy} if, for $\ell$ large enough, the image of $\rho_\ell$ is exactly $\SP_{2d_A}(\F_\ell)$ over $\kbar(\B)$. These properties of having big monodromy implicitly use the choice of a polarization, but do not depend on this choice.

We now return to our running assumptions: $A$ is an abelian variety over $k(\B)$, whose $k(\B)/k$-trace vanishes, and $\ell\neq \car(k)$ is a prime.

The aim of this section is to prove the following:

\begin{theorem}
	\label{thm:big_monodromy}
	Let $A$ be a polarized abelian variety over $k(\B)$ whose $k(\B)/k$-trace vanishes, and let $\ell\neq \car(k)$ be an odd prime which does not divide the degree of the polarization. Assume that the image of the Galois representation $\rho_\ell$ over $\kbar(\B)$ is the full symplectic group $\SP_{2d_A}(\F_\ell)$. Then the equivalent assumptions of Lemma~\ref{lem:CohCond} hold, and we have an injective map
	$$
	H^1(\B,\A[\ell]) \hookrightarrow \ker(N_{\X/\X^+}:\Pic(\X)[\ell] \to \Pic(\X^+)[\ell]),
	$$
	where $\X\to \B$ is the (geometrically  irreducible) cover of smooth projective $k$-curves with generic fiber $A[\ell]\setminus\{0\}$, and $\X\to \X^+$ is the quotient of $\X$ by the involution $P\mapsto -P$.
\end{theorem}

\begin{remark}
	\label{rmk:big_monodromy}
	\begin{enumerate}
		\item Theorem~\ref{thm:big_monodromy} is a higher dimensional analogue of \cite[Theorem~1.1,~1)]{gl22}. By combining it with Theorem~\ref{thm:main}, one obtains a bound for the rank of $A$ over $k(\B)$.
		\item When $A=E$ is a non-isotrivial elliptic curve over $k(\B)$, Igusa \cite[Theorem~3]{Igusa59} proved that, for all but finitely many $\ell$, the image of the Galois representation $\Gal(\overline{k(\B)}/\overline{k}(\B))\to \GL_2(\F_\ell)$ attached to $E[\ell]$ is $\SL_2(\F_\ell)=\SP_2(\F_\ell)$. Moreover, Cojocaru and Hall \cite[Theorem~1]{Cojocaru_Hall_2005} gave a precise bound $M_\B$, depending only on the genus of $\B$, such that for all primes $\ell> M_\B$, the image of the above Galois representation is $\SL_2(\F_\ell)$.
		\item Assume $k$ is a number field. If $\End(A)=\Z$ and $\dim(A)=2,6$ or is odd, then it was proved by Serre \cite[Lettre à Marie-France Vignéras du 10/2/1986]{serreIV} that $A$ has big monodromy. If $A$ is not isotrivial, then it follows from \cite[Proposition~4.1]{LST2019} that $A$ has big geometric monodromy.
		\item Assume $k$ is a finitely generated field, $\End(A)=\Z$ and $A$ has semistable reduction of toric dimension one at some closed point of $\B$. Then it was proved by Arias-de-Reyna, Gajda and Petersen \cite[Theorem~3.6]{AriasdeReyna_Gajda_Petersen_2013} that $A$ has big monodromy, extending a result by Hall in the global field case \cite[Theorem~1 and \S{}3, Remark]{Hall_2011}. If $A$ is not isotrivial, one deduces as previously that $A$ has big geometric monodromy.
		\item Although the assumptions of Theorem~\ref{thm:big_monodromy} may seem strong, these are generically satisfied. Indeed, it follows from the work of Deligne and Mumford \cite[5.12]{DM1969} that the moduli space of genus $g$ curves has big geometric monodromy. More precisely, the image of the adelic Galois representation attached to the Jacobian of a generic curve of genus $g$ (over $\kbar$) is the full symplectic group $\SP_{2g}(\hat{\Z})$. Principally polarized abelian varieties of dimension $d$ share the same property, as pointed out in \cite[Theorem~5.4]{LST2019}.
	\end{enumerate}
\end{remark}

The following lemma is a key step in proving Theorem~\ref{thm:big_monodromy}.

\begin{lemma}
	\label{lem:H1picard}
	Assume that $A$ has a polarization, defined over $k(\B)$, of degree coprime to $\ell$. Let $X\subset A[\ell]$ be a reduced closed $k(\B)$-subscheme, which is irreducible over $\kbar(\B)$. Assume that
	\begin{enumerate}
		\item[(1)] $X$ generates $A[\ell]$ as a group;
		\item[(2)] $A[\ell]$ has no nonzero rational point over $\kbar(\B)$;
		\item[(3)] $H^1(\Gal(\kbar\Kell/\kbar(\B)),A[\ell])=0$, where $\Kell$ is the $\ell$-th division field of $A$.
	\end{enumerate}
	Then the equivalent assumptions of Lemma~\ref{lem:CohCond} hold, and we have an injective map
	$$
	H^1(\B,\A[\ell]) \hookrightarrow \ker(N_{\X/\B}:\Pic(\X)[\ell] \to \Pic(\B)[\ell]),
	$$
	where $\X\to \B$ is the (geometrically irreducible) cover of smooth projective $k$-curves with generic fiber $X$.
Finally, if one replaces the first assumption above by
	\begin{enumerate}
		\item[(1')] $\ell$ is odd, $X$ generates $A[\ell]$ as a group, and is stable under the map $P\mapsto -P$,
	\end{enumerate}
then one can improve the conclusion as follows: we have an injective map
	$$
	H^1(\B,\A[\ell]) \hookrightarrow \ker(N_{\X/\B}:\Pic(\X)[\ell] \to \Pic(\X^+)[\ell]),
	$$
	where $\X\to \X^+$ is the quotient of $\X$ by the involution $P\mapsto -P$.
\end{lemma}

Let us observe that, since $X$ is irreducible over $\kbar(\B)$---hence over $k(\B)$---it consists of a single orbit under the action of $\Gal(\Kell/k(\B))$. As $X$ generates $A[\ell]$, the field of definition of all points of $X$ contains $\Kell$, hence is equal to it. Therefore, $\Kell$ is the Galois closure of $k(\X)/k(\B)$.

\begin{remark}
It follows from the work of Guralnick \cite[Theorem~A]{Guralnick99} that, if $A[\ell]$ is irreducible as a Galois module over $\overline{k}(\B)$ and if $\ell>2g+2$, then assumption (3) is satisfied (and also (2) for obvious reasons). More generally, (3) holds if $A[\ell]$ is semi-simple as a Galois module and if $\ell>2g+2$, which is an effective version of a classical result of Nori \cite[Theorem E]{Nori87}. We refer the reader to \cite{PL2023} for detailed proofs of these facts.
\end{remark}

\begin{proof}[Proof of Lemma~\ref{lem:H1picard}]
	Since $A^t[\ell]\simeq A[\ell]$, assumption (2) implies that the conditions of Lemma~\ref{lem:CohCond} are satisfied.
	The Weil pairing, composed with the isomorphism $A^t[\ell]\simeq A[\ell]$ induced by the polarization, induces a morphism of finite {\'e}tale $k(\B)$-group schemes (equivalently of Galois modules over $k(\B)$)
	\begin{equation}
		\label{eq:w:A[ell]}
		w: A[\ell] \to \Res_{k(X)/k(\B)} \mu_\ell, \quad P\mapsto e_\ell^\lambda(P,T)
	\end{equation}
	where $T$ denotes the generic point of $X$.
	This morphism is injective, because the Weil pairing is non-degenerate and $X$ generates $A[\ell]$ as a group. Moreover the sum of all points of $X$, which is the sum of all elements in a Galois orbit, is defined over $k(\B)$ hence is zero since $A[\ell]$ has no nonzero rational point over $\kbar(\B)$. It follows that the image of $w$ lies in the kernel of the norm map
	$$
	N_{k(X)/k(\B)} : \Res_{k(X)/k(\B)} \mu_\ell \to \mu_\ell.
	$$
	
	Since $\A[\ell]$, $\mu_\ell$ and $\Res_{\X/\B} \mu_\ell$ are Néron models of their generic fibers, it follows from the previous discussion that $w$ induces a morphism of {\'e}tale $\B$-group schemes (that we also denote by $w$)
	\begin{equation}
		\label{eq:embeddingA[ell]}
		w:\A[\ell] \hookrightarrow \Res_{\X/\B} \mu_\ell
	\end{equation}
	whose image is contained in the kernel of the norm map $N_{\X/\B}:\Res_{\X/\B} \mu_\ell \to \mu_\ell$.
	This induces a map $h^1w: H^1(\B,\A[\ell]) \rightarrow H^1(\X,\mu_\ell)$.
	Now, consider the commutative diagram
	$$
	\begin{CD}
		H^1(\B,\A[\ell]) @>h^1w>> H^1(\X,\mu_\ell) @>>> \Pic(\X)[\ell] \\
		@VVV @. @VVV \\
		H^1(\Bbar,\A[\ell]) @>h^1w_{\kbar}>> H^1(\overline{\X},\mu_\ell) @>\sim>> \Pic(\overline{\X})[\ell], \\
	\end{CD}
	$$
	in which the second row is the $\kbar$-analogue of the first one, and both horizontal maps on the right are induced by Kummer theory. Since $\X$ is a projective geometrically integral curve, we have that $\Gm(\X\otimes_k \kbar)=\kbar^\times$ which is $\ell$-divisible, and it follows that the Kummer map is an isomorphism over $\kbar$. Furthermore, the map $h^1w_{\kbar}$ is injective by Lemma~\ref{lem:h1w-inj}, and the vertical map on the left is injective, since under our assumptions $H^1(\B,\A[\ell])=H^1(\Bbar,\A[\ell])^{\Gal(\kbar/k)}$ by Lemma~\ref{lem:H1GalInv}.
	
	It follows that the composite map $H^1(\B,\A[\ell]) \to H^1(\X,\mu_\ell) \to \Pic(\X)[\ell]$ is injective. The image of this map lands in the kernel of the norm, since this is the case for the map $h^1w$. This concludes the proof of the first statement.
	
	Finally, let us replace assumption (1) by (1'). The Weil pairing satisfies $e_\ell^\lambda(P,T)\cdot e_\ell^\lambda(-P,T)=1$ by bilinearity. Since the map 
$$
H^1(\B,\A[\ell]) \hookrightarrow \Pic(\X)[\ell]
$$
is induced by the Weil pairing with the generic point of $\X$, its vanishes when one projects it, via the norm, to the quotient of $\X$ by the involution $P\mapsto -P$. Hence the result.
\end{proof}

\begin{lemma}
	\label{lem:h1w-inj}
	Under the assumptions of Lemma~\ref{lem:H1picard}, the map $h^1w_{\kbar}$ is injective.
\end{lemma}

\begin{proof}
	To shorten the notation, we assume that $k=\kbar$. Since $\A[\ell]$ is a Néron model, the restriction to the generic fiber map $H^1(\B,\A[\ell])\to H^1(k(\B),A[\ell])$ is injective. Hence it suffices to prove that the map
	$$
	H^1(k(\B),A[\ell]) \rightarrow H^1(k(X),\mu_\ell)
	$$
	is injective, where the groups involved are computed in Galois cohomology.
	
	For that purpose, we consider the exact sequence of $k(\B)$-finite group schemes (equivalently, of Galois modules over $k(\B)$)
	\begin{equation}\label{Galois-Exact-Sequence}
	\begin{CD}
		0 @>>> A[\ell] @>w>> \Res_{k(X)/k(\B)} \mu_\ell @>>> Q @>>> 0 \\
	\end{CD}
	\end{equation}
	where $w$ is the map in \eqref{eq:w:A[ell]} and $Q$ is the cokernel of $w$. By considering the action of the absolute Galois group of $k(\B)$, we obtain a long exact sequence of cohomology, whose first terms are
$$
0 \to A[\ell](k(\B)) \to \mu_\ell(k(X)) \to Q(k(\B)) \to H^1(k(\B),A[\ell]) \to H^1(k(X),\mu_\ell).
$$
The last map is injective if and only if the map $\mu_\ell(k(X)) \to Q(k(\B))$ is surjective. By considering the same exact sequence in the category of $\Gal(\Kell/k(\B))$-modules, and applying Galois cohomology again, we see that a sufficient condition for the surjectivity of $\mu_\ell(k(X)) \to Q(k(\B))$
is the vanishing of $H^1(\Gal(\Kell/k(\B)),A[\ell])$, which holds by assumption.
\end{proof}

\begin{lemma}
	\label{lem:h1=0}
	If $\ell$ is odd and $\Gal(\kbar\Kell/\kbar(\B))=\SP_{2d_A}(\F_\ell)$ then the assumptions (1'), (2) and (3) of Lemma~\ref{lem:H1picard} are satisfied by $X=A[\ell]\setminus\{0\}$.
\end{lemma}

\begin{proof}
For simplicity we let $d:=d_A$.
Since the group $\SP_{2d}(\F_\ell)$ acts transitively on $\F_\ell^{2d} \setminus \{0\}$, the subset $X=A[\ell]\setminus\{0\}$ is $\kbar(\B)$-irreducible. It remains to prove that $H^1(\SP_{2d}(\F_\ell),\F_\ell^{2d})=0$. For that, we observe that the subgroup $\{\pm I\}$ is in the center of $\SP_{2d}(\F_\ell)$, and, letting $N=\SP_{2d}(\F_\ell)/\{\pm I\}$, the corresponding inflation-restriction exact sequence reads
$$
0\to H^1(N,(\F_\ell^{2d})^{\{\pm I\}}) \to H^1(\SP_{2d}(\F_\ell),\F_\ell^{2d})) \to H^1(\{\pm I\},\F_\ell^{2d})^N \to \cdots
$$
Since $\ell$ is odd, it is clear that $(\F_\ell^{2d})^{\{\pm I\}}=0$, hence the first cohomology group vanishes. The third one also, since the orders of $\{\pm I\}$ and $\F_\ell^{2d}$ are coprime. This implies the vanishing of the cohomology group in the middle, as required.
\end{proof}

\begin{proof}[Proof of Theorem~\ref{thm:big_monodromy}]
According to Lemma~\ref{lem:h1=0}, the assumptions (1'), (2) and (3) of Lemma~\ref{lem:H1picard} are satisfied by the set $X=A[\ell]\setminus\{0\}$. The result follows.
\end{proof}


\subsection{Proof of Theorem~\ref{thm:jacobian}}
\label{subsec:hyperelliptic}


We now work under the assumptions of Theorem~\ref{thm:jacobian}.

Since $f$ has odd degree, the curve $C$ defined by the equation $y^2=f(x)$ has a unique point at infinity, which we denote by $\infty$. Let $X$ be the finite $k(\B)$-scheme defined by $f(x)=0$. By abuse of notation, we identify a point of $X$, i.e. a root $x_0$ of $f$, with the $2$-torsion point $(x_0,0)-\infty$ on $J$. It is a classical fact that these points generate $J[2]$. When $\deg(f)>3$, this set of points is strictly contained in $J[2]\setminus \{0\}$.

Under the assumptions of Theorem~\ref{thm:jacobian}, the trace of $J$ vanishes. Indeed, if the trace were nonzero, say of dimension $b>0$, then its $2$-torsion subgroup over $\kbar$ would be isomorphic to $(\Z/2\Z)^{2b}$ (since $\car(k)\neq 2$), hence $J[2]$ would have nonzero rational point over $\kbar(\B)$, equivalently $f$ would have roots defined over $\kbar(B)$, which contradicts the assumption of the Theorem that $f$ is irreducible over $\kbar(B)$.
For the same reason, the curve $\X$ is geometrically integral and its generic fiber $X$ generates $J[2]$. Moreover, the sum of all points of the generic fiber of $\X$ is zero. Therefore, the map $w$ induced by the Weil pairing \eqref{eq:w:A[ell]} is a closed embedding and factors as
\begin{equation}\label{Map_Induced_by_Weil_Pairing}
w: J[2] \to \ker(N_{k(X)/k(\B)}:\Res_{k(X)/k(\B)} \mu_2 \to \mu_2).
\end{equation}

Since $\car(k)\neq 2$, these two group schemes are finite \'etale. Let us compare their ranks: on the left-hand side, $J[2]$ has rank $2^{2g(C)}$ where $g(C)=\frac{\deg(f)-1}{2}$ is the genus of $C$. On the right-hand side, $\Res_{k(X)/k(\B)} \mu_2$ has rank $2^{\deg(f)}$, and the norm $N_{k(X)/k(\B)}$ is surjective since $[k(X):k(\B)]=\deg(f)$ is odd, hence the kernel of the norm has rank $2^{\deg(f)-1}$. To conclude, the ranks are the same, hence $w$ is an isomorphism.
	
	In fact, we have an exact sequence (of sheaves on the étale site of $\B$) between N{\'e}ron models of finite group schemes
	\begin{equation}
		\label{eq:A[ell]normkern}
		\begin{CD}
			0 @>>> \J[2] @>>> \Res_{\X/\B} \mu_2 @>N_{\X/\B}>> \mu_2 @>>> 0 \\
		\end{CD}
	\end{equation}
	(note that these N{\'e}ron models need not be finite themselves). Since N{\'e}ron models do not behave nicely with respect to exact sequences in general, the exactness of \eqref{eq:A[ell]normkern} requires detailed explanation: the left-exactness is because Néron models are direct images of their generic fibers on the smooth site, and the direct image functor is left-exact, so the sequence is left-exact on the smooth site, hence in particular on the étale site. The right-exactness is because the degree of $\X\to \B$ is odd. For the same reason, the map
	$$
	N_{\X/\B} : H^i(\B,\Res_{\X/\B} \mu_2)=H^i(\X,\mu_2) \rightarrow H^i(\B,\mu_2)
	$$
	is surjective for all $i\geq 0$. Hence from the long exact sequence attached to \eqref{eq:A[ell]normkern} we obtain short exact sequences for all $i\geq 0$:
	$$
	\begin{CD}
		0 @>>> H^i(\B,\J[2]) @>>> H^i(\X,\mu_2) @>N_{\X/\B}>> H^i(\B,\mu_2) @>>> 0. \\
	\end{CD}
	$$
	We deduce by the same argument as in the proof of Lemma~\ref{lem:H1picard} that $H^0(\B,\J[2])=0$, and that
	$$
	H^1(\B,\J[2]) = \ker(N_{\X/\B}:\Pic(\X)[2] \to \Pic(\B)[2]).
	$$
	The same holds over $\kbar$, and $J$ is self-dual, hence the assumptions of Theorem~\ref{thm:main} are satisfied, from which the result follows.

\begin{proof}[Proof of Proposition~\ref{explicitdescent}]
	The proof follows the same lines as \cite[Proposition~4.1]{GHL23}. According to Schaefer \cite[Theorem~1.2]{Schaefer1995}, the composition of the cohomological maps
	$$
	\begin{CD}
		J(k(\B))/2J(k(\B)) @>\delta>> H^1(k(\B),J[2]) @>h^1\omega>> H^1(k(\X),\mu_2), \\
	\end{CD}
	$$
	(where $\delta$ is the usual Kummer map) can be described as
	\begin{align*}
		J(k(\B))/2J(k(\B)) &\rightarrow k(\X)^\times/(k(\X)^\times)^2\simeq H^1(k(\X),\mu_2) \\
		[D] &\mapsto a_D
	\end{align*}
	where $a_D$ is the $x$-coordinate polynomial in the Mumford representation of $D$.
	
	Going through the proof of Lemma~\ref{lem:l-descent}, one can check that the Kummer map $\delta$ has values in the subgroup $H^1(\B\setminus \Sigma_2,\J[2])$, where $\Sigma_2\subset \B$ is of the set of $v\in\Sigma$ for which $(\Phi_v/2\Phi_v)(k_v)\neq 0$. It follows that the map $h^1\omega\circ \delta$ has values in $H^1(\X\setminus \Sigma_2^\X,\mu_2)$.
	We finally observe that $\phi_2$ is the composition of $h^1\omega\circ \delta$ with the natural morphism
	\begin{equation*}
		\begin{split}
			\kappa_2: H^1(\X\setminus \Sigma_2^\X, \mu_2) &\rightarrow \Pic(\X, \Q.\Sigma_2^\X)[2] \\
			h\in k(\X)^\times &\mapsto \frac{1}{2}\divisor(h).
		\end{split}
	\end{equation*}
	Therefore, $\phi_2$ is a morphism of groups. The injectivity of $\phi_2$ follows from the injectivity of $\delta$, the injectivity of $h^1\omega$ (Lemma~\ref{lem:h1w-inj}), and the elementary fact \cite[\S{}2]{GHL23} that the kernel of $\kappa_2$ is $k^\times/(k^\times)^2$, which does not contribute to the kernel of the norm map on the $H^1$'s, hence intersects trivially the image of $h^1\omega$.
\end{proof}


\subsection{Proof of Theorem~\ref{theoremd5}}

The main idea in the proof of Theorem~\ref{theoremd5}~(i) is to view the surface defined by $y^2=g(t)+x^5$ as a fibration in two different ways, by projecting onto the $x$-coordinate or the $t$-coordinate (yielding an elliptic and genus $2$ fibration, respectively). We then make a comparison between the two fibrations, taking advantage of well-known results in the theory of elliptic surfaces.
We begin with a proof of (ii).

\begin{proof}[Proof of Theorem~\ref{theoremd5}~(ii)]
We will construct a bijection between $R$ and a related set of points $R'$ on the elliptic curve $E'$ over $k(x)$ given by $y^2=g(t)+x^5$. We first note that the discriminant of $E'$ is a quadratic polynomial in $x^5$ not divisible by $x$ (since $g$ is squarefree). Then the discriminant of $E'$ has degree $10$, and factors of multiplicity at most $2$, from which it follows that the associated elliptic surface is a rational elliptic surface with all fibers irreducible.  In this case, it is known that the associated Mordell-Weil lattice is $E_8$ and there is a set $R'$ of $240$ points of the form $(t(x),y(x))\in E'(\kbar(x))$ with $t(x),y(x)\in \kbar[x]$ and $\deg t(x)\leq 2$, $\deg y(x)\leq 3$, corresponding to the $240$ root vectors in $E_8$ \cite[\S{}10]{Shioda90}.

As noted near the beginning of the proof of Theorem~\ref{theoremd5}~(i), the set $\Sigma_2^\X$ is empty.  Therefore, the $2$-descent map of Proposition~\ref{explicitdescent} can be described as
	\begin{align*}
		E'(\kbar(x))/2E'(\kbar(x)) &\rightarrow \Jac(\mathcal{X})[2](\kbar) \\
		(t(x),y(x)) &\mapsto \frac{1}{2}\divisor(t-t(x)).
	\end{align*}

Also, since $K_{\mathcal{X}}\sim 6\infty$ (see the remarks before Theorem \ref{theoremd5}), we have a bijective map 
\begin{align*}
	\Jac(\mathcal{X})[2](\kbar)&\to \{\text{theta characteristics of $\mathcal{X}$}\}\\
	[D]&\mapsto [D+3\infty].
\end{align*}
We can compose the above $2$-descent map with this latter map, and the restriction of the $2$-descent map to $R'$ can be seen to be $2$-to-$1$ from the description as root vectors in $E_8$ (for root vectors $u,v\in E_8$, it is easy to see that $u-v\in 2E_8$ if and only if $u=\pm v$). From the form of the points in $R'$ it follows that $\frac{1}{2}\divisor(t-t(x))+3\infty$ is effective. Now, the effective theta characteristics on the genus $4$ curve $\mathcal{X}$ consist of the odd theta characteristics along with the (unique) vanishing even theta characteristic $3\infty$ (see \cite[Theorem~2.8.8]{Kulkarni} and \cite[\S{}2.8]{Vakil}). Therefore, the image of $R'$ via the above composite map is precisely the set of $120$ odd theta characteristics of $\mathcal{X}$.  From the form of $E'$ one can see that, for each point $(t(x),y(x))\in R'$, we have $\deg t(x)=2$ and $\deg y(x)=3$.

Next, we construct a natural bijection between $R'$ and $R$.  Given $(t(x),y(x))\in R'$, let $a_D$ be the monic quadratic polynomial in $x$ of the form $x^2+bx+c(t), b\in \kbar, c(t)\in \kbar[t], \deg c(t)=1$, obtained by rescaling $t(x)-t$ (to be monic in $x$). Then the zeroes of $a_D=0$ determine the $x$-coordinates of two points $P,P'$ and the $y$-coordinates of $P,P'$ are determined by the equation $y=y(x)$, yielding the divisor class $[P+P'-2\infty]\in R$. Conversely, given $[D]\in R$, let $a_D=x^2+bx+c(t)$ be as in the definition of $R$, and let $b_D(t,x)$ be the other polynomial in the Mumford representation of $[D]$, defining appropriate $y$-coordinates. Since $a_D$ is linear in $t$, we can solve for $t$ obtaining $t=t(x)$, $\deg t(x)=2$, and substitute to find $y(x)=b_D(t(x),x)$. Then $(t(x),y(x))\in R'$.

Finally, the composite map $R\to R'\to  \{\text{odd theta characteristics of $\mathcal{X}$}\}$ clearly induces the map of Theorem~\ref{theoremd5}~(ii)  when $k=\kbar$, and we deduce the general result when $k$ is not algebraically closed by noting that the map appropriately respects the natural Galois action.
\end{proof}

\begin{proof}[Proof of Theorem~\ref{theoremd5}~(i)]
Let $X\to\mathbb{P}^1$ be the relatively minimal fibration associated to $C$, induced by $(x,y,t)\mapsto t$, whose general fiber is a smooth projective curve of genus $2$.  Using Liu's algorithm \cite{Liu94} and the Namikawa-Ueno classification \cite{NU73}, we find that the singular fibers consist of irreducible cuspidal curves (type [VIII-1]), one for each root of $g(t)$, and the singular fiber over $\infty$, which, by an appropriate change of coordinates, is seen to be of type [VIII-3]. The singular fiber over $\infty$ is given by $2B+3\Gamma_1+2\Gamma_2+\Gamma_3$, where all the curves are rational, and we have $B^2=-3$, $\Gamma_i^2=-2$, $i=1,2,3$, $B.\Gamma_1=2$, $\Gamma_1.\Gamma_2=1$, $\Gamma_2.\Gamma_3=1$, and all other intersection numbers are $0$. In particular, at each place of bad reduction the component group of the N\'eron model of $J$ is trivial, and thus $\Sigma_2^\X$ is empty. We now show that $\rk  J(\kbar(t))=8$, and moreover the Mordell-Weil lattice associated to $J(\kbar(t))$ is the $E_8$ lattice. We remark that since $g(\mathcal{X})=4$, it follows from Proposition \ref{explicitdescent} that $\rk J(\kbar(t))\leq 2g(\mathcal{X})=8$, and therefore in this case the $2$-descent bound of Proposition \ref{explicitdescent} (or Theorem \ref{thm:jacobian}) is sharp.

First, we give an explicit description of $X$. As is well known, the relatively minimal rational elliptic surface $X'$ associated to $E'$ can be obtained as the blowup of $\mathbb{P}^2$ at $9$ points $P_1,\ldots, P_9$, with associated exceptional divisors $E_1,\ldots, E_9$, where $E_9$ is the identity section and $E_1,\ldots, E_8$ give sections (over $\kbar$) corresponding to points $(t(x),y(x))\in E'(\kbar(x))$ with $t(x),y(x)\in \kbar[x]$ and $\deg t(x)=2$, $\deg y(x)=3$. The surfaces $X$ and $X'$ are birational, and we let $F$ be defined by $t=t_0$ in $X'$, corresponding to a fiber of $X$. Then from their descriptions, $F$ intersects each section $E_1,\ldots, E_8$ in $2$ points (with multiplicity). Moreover, the fibers $F'$ of $X'\to \mathbb{P}^1$ correspond to members of a pencil $\Lambda$ of cubics through $P_1,\ldots, P_9$ and we also easily find $F'.F=2$ in $X'$. It follows that the image of $F$ via the map $X'\to\mathbb{P}^2$ is a sextic curve with double points at $P_1,\ldots, P_8$, and we are in Case (A) of \cite{Kitagawa10}. From the proofs of \cite[Theorem 3.1, Theorem 3.16]{Kitagawa10}, we can obtain $X$ from $X'$ by blowing down $O'=E_9$ to obtain a surface $Y$ and map $\pi':X'\to Y$, with the pencil of curves on $Y$ corresponding to $\Lambda$ now passing through a single point, and then blowing up $4$ (possibly infinitely near) points on $Y$ to obtain $X$ as the blowup $\pi:X\to Y$. From \cite{Kitagawa10} (or keeping track of the number of blowups and blowdowns), $X$ has Picard number $\rho(X)=13$, and by the Shioda-Tate formula,
		\begin{align*}
			\rk J(\kbar(t))=\rho(X)-2-\sum_v (m_v-1)=13-2-3=8,
		\end{align*}
where $m_v$ denotes the number of components in the fiber of $X\to\mathbb{P}^1$ over a closed point $v\in \mathbb{P}^1$. Thus $\rk  J(\kbar(t))=8$ as claimed.

We now derive the Mordell-Weil lattice structure on $J(\kbar(t))$. Let $T$ (respectively $T'$) be the trivial lattices in $\NS(X)$ (respectively in $\NS(X')$) generated by the irreducible components of fibers and the section $O$ (respectively $O'$) corresponding to the point at infinity on $C$ (respectively on $E'$). Then due to Shioda \cite[Theorem~2.1 and Lemma~2.3]{Shioda92b}, we have isomorphisms $J(\kbar(t))\cong \NS(X)/T$, $E'(\kbar(x))\cong \NS(X')/T'$, and corresponding ``splitting'' homomorphisms $\phi:J(\kbar(t))\to \NS(X)\otimes \mathbb{Q}$ and $\phi':E'(\kbar(x))\to \NS(X')\otimes \mathbb{Q}$ which yield the Mordell-Weil lattice structures via intersection theory (e.g., $\langle P,Q\rangle=-(\phi(P).\phi(Q))$, $P,Q\in J(\kbar(t))$). The map $\phi$ is characterized by the properties
	\begin{align*}
		\phi(P)\equiv D_P \pmod{T\otimes \mathbb{Q}}, \quad \phi(P)\perp T,
	\end{align*}
	where $D_P$ is a horizontal divisor on $X$ corresponding to $P$. A similar statement holds for $\phi'$.
	
	Let $P\in R, P'\in R'$ be elements that correspond under the natural bijection $R\to R'$. We now compare $\phi(P)$ and $\phi'(P')$. Let $D_{P'}$ be the section of $X'\to\mathbb{P}^1$ corresponding to $P'$, let $O'$ be the identity section of $X'$, and let $F'$ be a fiber. Since every singular fiber of $X'$ is irreducible, we have the formula \cite[p.~108]{SchSch10}
	\begin{align*}
		\phi'(P')=D_{P'}-O'-(D_{P'}.O'+\chi(X'))F'=D_{P'}-O'-F',
	\end{align*}
	using that $D_{P'}.O'=0$ and $\chi(X')=1$. Since $D_{P'}$ does not intersect $O'$, we identify $D_{P'}$ with its image in $Y$. Let $F''=\pi'(F')$ be the image in $Y$ of a fiber $F'$ on $X'$. Then since $X'$ is obtained as a blowup at a point in $F''$, we find $\phi'(P')=\pi'^*(D_{P'}-F'')$.

We claim that
		\begin{align}
			\label{phiformula}
			\phi(P)=\pi^*(D_{P'}-F'')=\pi^*(D_{P'})-\pi^*(F'').
		\end{align}
Since the divisor $\pi^*(D_{P'})$ is associated to $\phi(P)$, we need only show that $\pi^*F''$ is a linear combination of the identity section $O$ and fiber components, and that $\pi^*(D_{P'})-\pi^*(F'')$ is orthogonal to $O$ and to every fiber component. For ease of notation, we identify $E_i$ with the divisor $\pi^*(\pi'_*(E_i))$ on $X$, $i=1,\ldots, 8$, and we let $E_{10},\ldots, E_{13}$ be the exceptional divisors in the four successive blowups to obtain $X$ from $Y$ (identified with their pullback to $X$). We let $\ell$ be any divisor in the class of (the pullback of) a line on $\mathbb{P}_k^2$.  Then with an appropriate ordering, a calculation shows that
	\begin{align*}
		B&\sim 3\ell-\sum_{i=1}^8E_i-2E_{10},\\
		\Gamma_i&\sim E_{9+i}-E_{10+i} \text{ for } i=1,2,3 \text{ and}\\
		O&=E_{13}.
	\end{align*}
	It follows that
	\begin{align*}
		\pi^*F''\sim 3\ell-\sum_{i=1}^8E_i\sim B+2\Gamma_1+2\Gamma_2+2\Gamma_3+2O.
	\end{align*}
	Now we easily find $D_{P'}.F''=(\pi^*D_{P'}).B=(\pi^*F''.B)=1$, $(\pi^*D_{P'}).O=(\pi^*F'').O=0$, and $(\pi^*D_{P'}).\Gamma_i=(\pi^*F'').\Gamma_i=0$ for $i=1,2,3$. Then the required orthogonality for $\pi^*(D_{P'}-F'')$ holds and equation \eqref{phiformula} follows.

Since birational pullbacks preserve intersection numbers, we find that
	\begin{align*}
		\langle P,Q\rangle=-(\phi(P).\phi(Q))=-(\phi'(P').\phi'(Q'))=\langle P',Q'\rangle,
	\end{align*}
	for corresponding pairs $(P,Q)\in R^2$ and $(P',Q')\in R'^2$. Moreover, it follows from the above calculations that $\NS(X)/T\cong \NS(X')/T'$, hence $J(\kbar(t))\cong  E'(\kbar(x))\cong \mathbb{Z}^8$ (in particular, $J(\kbar(t))$ is torsion-free). Since $R'$ generates the Mordell-Weil lattice of $X'$ and it is the $E_8$ lattice, it follows that $R$ generates the Mordell-Weil lattice of $X$ and it is the $E_8$ lattice. 
\end{proof}

\begin{proof}[Proof of Theorem~\ref{theoremd5}~(iii)]
We first assume $k=\kbar$ is algebraically closed. From the shape of the equation of $C$, we see that its Jacobian $J$ contains a cyclic subgroup of order $5$, generated by a divisor class of the form $[P-\infty]$ with $x(P)=0$. This defines a $5$-isogeny $J\to J'$ for some $J'$. By a construction similar to that of Proposition~\ref{explicitdescent}, one can describe the descent map with respect to the dual $5$-isogeny $J'\to J$ as follows:
	\begin{align*}
		R&\to E[5](k)\setminus \{0\},\\
		[D]&\mapsto \frac{1}{5}\divisor(c_D).
	\end{align*}
	
	Note that the target of this descent map consists of classical divisor classes (no rational coefficients involved), since as noted near the beginning of the proof of Theorem~\ref{theoremd5}~(i), the component group of the N\'eron model of $J$ is trivial.

We shall describe this $10$-to-$1$ map by explicitly giving the fiber of this map over a point of $E[5](k)\setminus \{0\}$. Equivalently, we explicitly give the points of $R'$ corresponding to points in such a fiber, as it is slightly easier to verify the calculations in this case. Since the field $k$ is algebraically closed, we may assume, after a change of variables, that $g(t)$ and $E$ are given by the equation
	\begin{equation*}
		\begin{split}
		E:y^2=g(t) &=t^3 + (-27u^4 + 324u^3 - 378u^2 - 324u -27)t\\
		&+ (54u^6 - 972u^5 + 4050u^4 + 4050u^2 + 972u + 54)
		\end{split}
	\end{equation*}
	for some parameter $u\in k$,
	with $5$-torsion point $(3u^2 - 18u + 3, -108u)$; this is obtained by putting the universal elliptic curve over $X_1(5)$ in Weierstrass form \cite[Table 3, Entry 13]{Kubert76}. Let $w$ be an algebraic number satisfying
	\begin{align*}
		w^5=1296\frac{5\sqrt{5}+11}{2u -5\sqrt{5}- 11}.
	\end{align*} 
	Then by explicit calculation we find the point $(t(x),y(x))\in R'$ on $E'$ where
	\begin{align*}
		t(x)&=\frac{1}{72}(3-\sqrt{5})w^2x^2 + wx + 3u^2 - 18u + 3,\\
		y(x)&=\frac{1}{216}(2-\sqrt{5})w^3x^3 + \frac{1}{24}((\sqrt{5} - 3)u + 7-\sqrt{5})w^2x^2 - (3u -3)wx - 108u,
	\end{align*} 
	and one verifies (using Magma \cite{Magma}) that the corresponding point in $R$ maps to the given $5$-torsion point. We have $5$ choices of the root $w$, leading to $5$ such points in $R'$ (or $R$), and changing $\sqrt{5}$ to $\sqrt{-5}$ everywhere yields $5$ more points. Since $|R|=240$ and $|E[5](k)\setminus \{0\}|=24$, we see that the map $R\to E[5](k)\setminus \{0\}$ must be exactly $10$-to-$1$.
	
	Finally, we deduce the analogous result when $k$ is not algebraically closed as the relevant map appropriately respects the natural Galois action.
\end{proof}


\subsection{Proof of Theorem~\ref{theoremd6}}

Since the techniques are similar to the ones in the proof of Theorem~\ref{theoremd5}, and the primary role of Theorem~\ref{theoremd6} is to illustrate the different behavior that can occur in the even case, we only sketch some details of the proof.

\begin{proof}[Proof of Theorem~\ref{theoremd6}]
We first give a quick proof that $J(\kbar(t))\cong \mathbb{Z}^6\oplus \mathbb{Z}/3\mathbb{Z}$. We exploit the well-known fact that the curve $C$ is bielliptic. Define the two elliptic curves over $k(t)$:
$$
	E_1:y^2 = x^3+g(t) \quad\text{and}\quad E_2:y^2 = x^3+g(t)^2.
$$
Then $C$ admits the two maps (both of degree $2$):
\begin{align*}
C &\to E_1, & C &\to E_2, \\
(x,y) &\mapsto (x^2,y) & (x,y) &\mapsto (g(t)/x^2,g(t)y/x^3),
\end{align*}
and this induces a $(2,2)$ isogeny between $J$ and $E_1\times E_2$. It follows from Oguiso-Shioda's classification of rational elliptic surfaces \cite[Main~Theorem]{OS91} (see also \cite[Table~2, Types (3,0,1,0,0), (0,3,0,0,0)]{GHL23}) that
\begin{align*}
	E_1(\kbar(t))&\cong \mathbb{Z}^4,\\
	E_2(\kbar(t))&\cong \mathbb{Z}^2\oplus \mathbb{Z}/3\mathbb{Z}.
\end{align*}
Since $x^6+g(t)$ is irreducible over $\kbar(t)$, it follows that $J(\kbar(t))$ has trivial $2$-torsion, and thus we find that 
\begin{align*}
	J(\kbar(t))\cong E_1(\kbar(t))\oplus E_2(\kbar(t))\cong \mathbb{Z}^6\oplus \mathbb{Z}/3\mathbb{Z}.
\end{align*}
Looking at the principal divisor associated to $y-x^3$, we find that $[\infty^--\infty^+]$ generates the $3$-torsion subgroup of $J(\kbar(t))$.

The relatively minimal fibration $X\to\mathbb{P}^1$ associated to $C$ has $3$ singular fibers of type [V] (one for each root of $g$) and one singular fiber of type [II] (over $\infty$) in the classification of \cite{NU73}. Let $E'$ be the elliptic curve over $\kbar(x)$ defined by $y^2=g(t)+x^6$, and let $X'$ the associated relatively minimal rational elliptic surface. As in the proof of Theorem~\ref{theoremd5}, $X$ can be obtained by blowing down $X'$ once and then blowing up $4$ times, and an explicit calculation (using Magma \cite{Magma}) of $\NS(X)/T$ (and relevant intersection pairings) shows that the Mordell-Weil lattice is $E_6^*$. 

Lastly, we consider the set $R$. As in the proof of Theorem~\ref{theoremd5}, one can show that the Mordell-Weil lattice associated to $E'(\kbar(x))$ is $E_8$, and there is a set $R'$ of $240$ points of the form $(t(x),y(x))\in E'(\kbar(x))$ with $t(x),y(x)\in \kbar[x]$ and $\deg t(x)\leq 2$, $\deg y(x)\leq 3$, corresponding to the $240$ root vectors in $E_8$. The sets $R$ and $R'$ are not in bijection, however, as the $6$ points $(\alpha, \pm x^3)\in E'(\kbar(x))$, where $\alpha$ is a root of $g(t)$, do not yield a point in $R$ (as the first coordinate is constant). Excluding these points from $R'$, essentially the same construction as in the proof of Theorem~\ref{theoremd5} gives a bijection and shows that $|R|=|R'|-6=234$. More precisely, identifying $R$ in $\NS(X)$ appropriately, one can explicitly compute the Mordell-Weil lattice pairing on elements of $R$ (using Magma \cite{Magma}) and verify the remaining claims of the theorem. 
\end{proof}

\begin{remark}
	\label{remdescent}
	It may be tempting to consider the map $\phi_2$ defined in Proposition~\ref{explicitdescent} for an arbitrary hyperelliptic curve, regardless of the degree of the polynomial $f$. In the number field case, it was observed by Cassels \cite[\S{}5]{cassels} that when $f$ has degree $6$ this map need not be the $2$-descent map. The same occurs in the function field case, as we shall illustrate.
	In the setting of Theorem~\ref{theoremd6}, let $\mathcal{X}$ be the smooth projective curve over $k$ defined by $x^6+g(t)=0$. In the notation of Proposition~\ref{explicitdescent}, we have $\Sigma_2^\X=\emptyset$, so the analogue of $\phi_2$ is the map
	\begin{align*}
		J(\kbar(t)) &\rightarrow \Jac(\mathcal{X})(\kbar)[2], \\
		[D] &\mapsto \frac{1}{2}\divisor(a_D(x)).
	\end{align*}
	When restricted to the set $R$, this map is $2$-to-$1$ and its image can be naturally identified with odd theta characteristics on $\mathcal{X}$, excluding the $3$ odd theta characteristics corresponding to the $6$ points excluded from $R'$ in the proof of Theorem~\ref{theoremd6}. On the other hand, the map $J(\kbar(t))\to J(\kbar(t))/2J(\kbar(t))$ is $6$-to-$1$ on the subset $R_1$ of $R$ from Theorem~\ref{theoremd6}. Thus, it is impossible for the map $\phi_2$ of Proposition~\ref{explicitdescent} to be the $2$-descent map in this case.
\end{remark}


\section*{Acknowledgments}
The first author acknowledges the support of the Pacific Institute for Mathematical Sciences (PIMS) while he was a PIMS postdoctoral fellow at the University of Lethbridge. The second author was supported in part by the CIMI Labex. The third author was supported in part by NSF grant DMS-2302298.





\bibliographystyle{amsalpha}
\bibliography{biblioRBF.bib}



\end{document}